\documentclass[12pt,twoside]{amsart}
\usepackage{mabliautoref}
\usepackage{amssymb,amsthm,amsmath}
\RequirePackage[dvipsnames,usenames]{xcolor}
\usepackage{hyperref}
\usepackage{mathtools}
\usepackage[abbrev,alphabetic]{amsrefs}
\usepackage[all]{xy}
\usepackage{tikz}
\usepackage{tikz-cd}
\usepackage{systeme}

\hypersetup{
	bookmarks,
	bookmarksdepth=3,
	bookmarksopen,
	bookmarksnumbered,
	pdfstartview=FitH,
	colorlinks,backref,hyperindex,
	linkcolor=Sepia,
	anchorcolor=BurntOrange,
	citecolor=MidnightBlue,
	citecolor=OliveGreen,
	filecolor=BlueViolet,
	menucolor=Yellow,
	urlcolor=OliveGreen
}

\newtheorem{step}{Step}

\makeatletter
\newcommand{\newreptheorem}[2]{\newtheorem*{rep@#1}{\rep@title}\newenvironment{rep#1}[1]{\def\rep@title{#2 \ref*{##1}}\begin{rep@#1}}{\end{rep@#1}}}
\makeatother
\makeatletter
\@namedef{subjclassname@2020}{%
	\textup{2020} Mathematics Subject Classification}
\makeatother

\DeclareMathOperator{\Spec}{Spec}
\DeclareMathOperator{\Proj}{Proj}

\newcommand{\bQ}{\mathbb{Q}}

\newcommand{\MO}{\mathcal{O}}
\newcommand{\red}{\textup{red}}

\newcommand{\coeff}{\textup{coeff}}

\newcommand{\WDiv}{\textup{WDiv}}
\newcommand{\Pic}{\textup{Pic}}
\newcommand{\NS}{\textup{N$^1$}}
\newcommand{\Cl}{\textup{Cl}}

\newcommand{\SB}{\mathbf{SB}}

\newcommand{\Bf}[1]{\mathbf{#1}}

\newcommand{\Ht}[1]{H^{i}_{t}}

\newcommand{\ox}[1][X]{\mathcal{O}_{#1}}

\newcommand{\stacks}[1]{\cite[\href{https://stacks.math.columbia.edu/tag/#1}{Tag #1}]{stacks-project}}

\newtheorem{cas}{Case}

\usepackage{xcolor}

\title[Semiampleness for klt Calabi--Yau surfaces]
{Semiampleness for Calabi--Yau surfaces in positive and mixed characteristic} 
\author{Fabio Bernasconi and Liam Stigant} 
\subjclass[2020]{14E30, 14G17, 14J20, 14J27.}
\keywords{minimal model program, Calabi--Yau surfaces, abundance conjecture, positive characteristic.}

\address{\'Ecole Polytechnique F\'ed\'erale de Lausanne, Chair of Algebraic Geometry
	(B\^atiment MA), Station 8, CH-1015 Lausanne} 
\email{fabio.bernasconi@epfl.ch}

\address{Department of Mathematics, Imperial College London, 180 Queen's Gate, 
	London SW7 2AZ, UK} 
\email{l.stigant18@imperial.ac.uk}

\begin{document}
	
	\begin{abstract}
		In this note we prove the semiampleness conjecture for klt Calabi--Yau surface pairs over an excellent base ring.
		As applications we deduce  that generalised abundance and Serrano's conjecture hold for surfaces. Finally, we study the semiampleness conjecture for Calabi--Yau threefolds over a mixed characteristic DVR.
	\end{abstract}
	\maketitle
	
	\tableofcontents
	
	\section{Introduction}
	
	The abundance conjecture predicts that the canonical divisor $K_X$ of a minimal model $X$ with Kawamata log terminal (in short, klt) or log canonical (lc) singularities is semiample and it one of the most important conjectures of the Minimal Model Program (MMP). 
    Abundance is known to hold for surfaces over fields of arbitrary characteristic \cite{tanaka2020abundance, Pos21} and threefolds in characteristic 0 by Kawamata and Miyaoka (see \cite{FA} for references), but in higher dimensions even the effectivity of a multiple of $K_X$ (the so-called non-vanishing conjecture) is still an open problem.
	For threefolds in positive and mixed characteristic, various special cases have been proven in \cite{DW19, Zha19, Zha20, XZ19,Wit21, BBS} but the general conjecture is still unanswered.
	
	A generalised form of abundance is expected to hold for $K$-trivial varieties and their log generalisations (see \emph{e.g.} \cite[Conjecture 51]{Kol15}).
	
	\begin{conjecture}[Semiampleness conjecture on klt Calabi--Yau pairs]\label{c-num-semiample}
		Let $(X, \Delta)$ be a projective klt pair of dimension $n$ over a field $k$ such that $K_X+\Delta \equiv 0$.
		Let $L$ be a nef $\mathbb{Q}$-Cartier $\mathbb{Q}$-divisor on $X$.
		Then $L$ is num-semiample, i.e. there exists a $\bQ$-Cartier  $\bQ$-divisor $L'$ on $X$ such that $L \equiv L'$ and $L'$ is semiample.
	\end{conjecture}
	
	If $k$ is a field of characteristic $0$, a thorough discussion of this conjecture and its connections to the MMP can be found in the work of Lazi\'c and Peternell \cite{LP20, LP20II}. Under the same hypothesis on $k$, \autoref{c-num-semiample} has been proven for projective surfaces \cite[Theorem 8.2]{LP20}, certain classes of Calabi--Yau threefolds \cite{LOP, LS20} and hyperk\"{a}hler fourfolds \cite{DHMV22}. Moreover, some partial results for compact K\"{a}hler surfaces are obtained in \cite{FT18}.
	
	The aim of this note is to confirm the conjecture for projective surfaces over arbitrary, possibly imperfect, fields of positive characteristic.
	
	\begin{theorem}[see \autoref{t-num-semiampleness}]\label{t1-semi-ampl-CY}
	The semiampleness conjecture holds for klt Calabi--Yau surface pairs over a field $k$ of characteristic $p>0$.
	\end{theorem}
	
	In \cite{LP20, HL20} the authors propose a further generalisation of the abundance and semiampleness conjectures.
	
	\begin{conjecture}[Generalised abundance conjecture]\label{c-gen-ab}
	Let $(X, B)$ be a projective klt pair over a field $k$
	such that $K_X + B$ is pseudo-effective and let $M$ be a nef $\mathbb{Q}$-Cartier $\mathbb{Q}$-divisor on
	$X$. 
	If $K_X + B + M$ is nef, then it is num-semiample.
\end{conjecture}
	
Following, in part, ideas of \cite{LP20} we give a direct proof of the generalised abundance conjecture for excellent surfaces
by reducing to the semiampleness conjecture on klt Calabi--Yau surfaces over a field of characteristic $p>0$.

    \begin{theorem}[See \autoref{c-abundance-general-exc-surf}]\label{t-gen-abund-2}
	Let $\pi \colon X\to T$ be a projective $R$-morphism of quasi-projective integral normal schemes over $R$.
	Suppose that $(X,B)$ is a klt surface such that
	\begin{enumerate}
		\item $K_{X}+B$ is pseudo-effective over $T$;
		\item $M$ is a nef $\mathbb{Q}$-Cartier $\mathbb{Q}$-divisor over $T$;
		\item $L:=K_X+B+M$ is nef $\mathbb{Q}$-Cartier  $\bQ$-divisor over $T$.
	\end{enumerate} 
	Then $L$ is num-semiample over $T$.
    \end{theorem}

We note that similar results have also been obtained in the context of generalised surface pairs in characteristic 0 by Han-Liu \cite{HL20} and following their strategy we characterise when generalised abundance fails if $K_X+B$ is not pseudo-effective (\autoref{c-failure-gen-ab}).
As an application thereof, together with a careful analysis over imperfect fields, we provide a proof of Serrano's conjecture \cite{Ser95} for klt surfaces and suitable threefolds in the positive and mixed setting.

\begin{corollary}[\autoref{Serrano}, \autoref{Serrano-3}] \label{easy-Serrano}
    Let $R$ be an excellent ring of finite Krull dimension with dualising complex. Let $X \to T$ be a projective contraction of quasi-projective $R$-schemes. Suppose that $(X,B)$ is klt, $-(K_{X}+B)$ is strictly nef and further that
    
    \begin{enumerate}
        \item $X$ has dimension $2$; or
        \item $X$ has dimension $3$, $\dim T >0 $ and the closed points of $R$ have residue fields of characteristic $p=0$ or $p > 5$ 
    \end{enumerate}
    
    Then $-(K_{X}+B)$ is ample.
\end{corollary}

Using Serrano's conjecture we then show in \autoref{t-num-non-van} that the numerical non-vanishing conjecture of \cite{HL20} holds for generalised klt surface pairs even if $K_{X}+B$ is not pseudo-effective. 
We recall that the numerical non-vanishing conjecture is still open even for threefolds over the complex numbers $\mathbb{C}$ (see \cite{LPTX22} for recent progress).

To commplete the picture, we also show prove that generalised abundance holds for generalised lc pairs if the b-nef part is b-semiample \autoref{lc-abund} in positive and mixed characteristic. The main difficulty here is the lack of Bertini-type theorems.

Recently, a large part of the MMP for threefolds in mixed characteristic $(0,p>5)$ has been established in \cite{witaszek2020keels, bhatt2020, takamatsu2021minimal}. 
We conclude by showing an application of the semiampleness conjecture for surfaces to arithmetic klt CY threefolds.

\begin{theorem}\label{t-CY-3folds}
    Let $R$ be an excellent DVR with residue field $k$ of positive characteristic $p>5$.
    Let $\pi \colon (X,B) \to \Spec(R)$ be a projective dominant morphism such that $(X,B)$ is klt and $K_X+B \equiv 0$ over $R$.
    
    If  $L$ is a nef $\mathbb{Q}$-Cartier $\mathbb{Q}$-divisor over $R$, then $L$ is num-semiample over $R$.
\end{theorem}
	
	\textbf{Acknowledgments.}
	The authors thank P.~Cascini, S. Filipazzi, V. Lazić and R.~Svaldi for interesting discussions and comments on the topic of this note. 
	We are grateful to the referee for reading the manuscript carefully and suggesting modifications which improved the readability of the article.

	FB is supported by the NSF under grant number DMS-1801851 and LS is grateful to the EPSRC for his funding.
	
	\section{Preliminaries}
	
	\subsection{Notation}
	
	\begin{enumerate}
		\item In this article, a base ring $R$ will always denote an excellent domain of finite Krull dimension admitting a dualizing complex $\omega_R^{\bullet}$. We assume $\omega_R^{\bullet}$ is normalised as explained in \cite{bhatt2020}.
		\item For a field $k$ we denote by $k^{\text{sep}}$ (resp.~ $\overline{k}$) a separable (resp.~ an algebraic) closure of $k$.
		\item For an integral scheme $X$ with generic point $\eta$, its function field $k(X)$ is the field $\mathcal{O}_{X,\eta}$.
		\item If $X$ is an $\mathbb{F}_p$-scheme, we denote by $F \colon X \to X$ its (absolute) \emph{Frobenius morphism}. We say $X$ is \emph{$F$-finite} if $F$ is a finite morphism.
		\item For a field $k$, 
		we say that $X$ is a {\em variety over} $k$ or a $k$-{\em variety} if 
		$X$ is an integral scheme that is separated and of finite type over $k$. 
 		\item Given a scheme $X$, we denote by $X_{\text{red}}$ the reduced closed subscheme of $X$ underlying the same topological space (see \stacks{01IZ}).
		\item  We say that $(X, \Delta)$ is a \emph{log pair} if $X$ is a normal excellent integral
		pure $d$-dimensional Noetherian scheme with a dualizing complex, $\Delta$ is an effective $\mathbb{Q}$-divisor and $K_X+\Delta$ is $\mathbb{Q}$-Cartier. The dimension of $(X,\Delta)$ is the total dimension of $X$.
		\item We will follow \cite{kk-singbook} and \cite[Section 2.5]{bhatt2020} for the definition of singularities of log pairs (such as klt, lc).
		\item We refer to \cite[Section 2.5]{bhatt2020} and \cite{Laz04} for notions of positivity (such as big, nef, pseudo-effective) for $\mathbb{Q}$-Cartier $\mathbb{Q}$-divisors, relative to a projective morphism of separated Noetherian schemes.
		\item Let $L$ be a Cartier divisor on an integral scheme $X$ of finite type over $R$. 
		We denote by $\text{Bs}(L)$ the base locus of $L$, considered as a closed reduced subscheme of $X$.
		We denote by $\textbf{SB}(L):= \bigcap_{m > 0}\text{Bs}(mL)$ the \emph{stable base locus} of $L$ over $R$.
		\item A morphism $f \colon X \to Y$ of normal schemes is called a \emph{contraction} if $f$ is proper and $f_*\mathcal{O}_X=\mathcal{O}_Y$.
	\end{enumerate}

\subsection{Num-semiample divisors}

In this section, we fix $S$ to be a Noetherian excellent base scheme.
Given a proper scheme $X$ over $S$, we define a \emph{curve} in $X$ over $S$ to be an integral closed sub-scheme $C \subset X$  of dimension $1$ such that $C$ is proper over some closed point $s \in S$.
If it is clear from the context, we will omit to mention $S$.

\subsubsection{Nef and num-semiample}

The notion of nefness is numerical, while semiampleness is not as the example of a torsion non-trivial line bundle on an elliptic curve shows. 
The notion of numerical semiampleness is an interpolation between the two: while remaining a numerical condition, it implies the existence of a contraction morphism to a scheme.

\begin{definition}
	Let $X$ be a proper $S$-scheme. A $\bQ$-Cartier $\bQ$-divisor $L$ on $X$ is said to be \textit{semiample} (resp.~ \textit{num-semiample}) over $S$ if there exists a proper contraction $f \colon X \to Z$  of $S$-schemes and an ample $\bQ$-Cartier $\bQ$-divisor $A$ on $Z$ such that $L \sim_{\mathbb{Q}} f^*A$ (resp.~ $L \equiv f^*A$) over $S$.
\end{definition} 

Clearly a num-semiample  $\bQ$-Cartier $\bQ$-divisor is nef but it is easy to construct nef divisors which are not num-semiample (see \cite[Example 6.1]{LP20}).
Strictly nef divisors will appear frequently in our proofs.

\begin{definition}
	Let $X$ be a projective $S$-scheme and let $L$ be a $\mathbb{Q}$-Cartier $\mathbb{Q}$-divisor on $X$. We say $L$ is \emph{strictly nef} over $S$ if for every curve $C$ over $S$, we have $L \cdot C>0$.
\end{definition}
Note that the sum of a nef and a strictly nef line bundle is strictly nef.
We recall the definition of numerical dimension for nef divisors.

\begin{definition}
	Let $X$ be a normal projective variety defined over a field $k$ and let $L$ be a nef $\mathbb{Q}$-Cartier $\mathbb{Q}$-divisor.
	The \emph{numerical dimension} of $L$ is defined as 
	$$ \nu(L) := \text{max} \left\{ d \in \mathbb{Z}_{\geq 0} \mid L^d \not \equiv 0  \right\} $$
\end{definition}

\subsubsection{Descent of relatively numerically trivial divisors}

We collect some results on descent of trivial divisors.
We recall a descent for divisors which are $\mathbb{Q}$-linearly trivial on the generic fibre used successfully in \cite{CT20, Wit21}.

\begin{proposition}\label{p-descent-equidimensional}
 Let $f \colon X \to Z$ be a proper contraction between normal projective $S$-schemes, where $Z$ is 1-dimensional.
 Let $L$ be a nef $\mathbb{Q}$-Cartier  $\bQ$-divisor on $X$ such that $L|_{X_{k(Z)}} \sim_{\mathbb{Q}} 0$, where $X_{k(Z)}$ is the generic fibre. 
 Then there exists a $\mathbb{Q}$-Cartier $\bQ$-divisor $L_Z$ on $Z$ such that $L \sim_{\mathbb{Q}} f^*L_Z$.
\end{proposition}

\begin{proof}
It is sufficient to note that hypothesis of \cite[Lemma 2.17]{CT20} are easily verified when the base has dimension 1.
\end{proof}

We now discuss numerical descent of numerically trivial nef divisors. 
In characteristic $0$ several results are proven in \cite{Leh15} and \cite[Lemma 3.1]{LP20} and their analogues for threefolds in positive characteristic are discussed in \cite[Section 5]{BW17}.
We begin with the case of numerical descent over a curve. 

\begin{lemma}\label{l-descent-num-triv-base}
    Let $f \colon X \to Y$ be a projective dominant morphism of integral $S$-quasi-projective excellent schemes. Let $L$ be a $\mathbb{Q}$-Cartier $\mathbb{Q}$-divisor and suppose that 
	\begin{enumerate}
		\item $L$ is nef;
		\item $L|_{X_{k(Y)}} \equiv 0$; and
		\item $Y$ has dimension $1$.
	\end{enumerate}
	Then $L \equiv_f 0$. 
\end{lemma}

\begin{proof}
	We can freely replace $Y$ with its normalisation and $X$ by the normalisation of the corresponding fibre product. Then, after taking a Stein factorisation, we may suppose $Y$ is regular and $X \to Y$ is a flat contraction.
	Let $C$ be a curve contracted to a closed point $y\in Y$. It suffices to show $L \cdot C=0$. Since $X \to Y$ is flat, the fibres are of pure dimension $d$. If $d=1$ then $L \cdot C=0$ by the argument of \cite[Lemma 5.3]{BW17}. 
	
	Otherwise if $d > 2$, we cut with a very ample Cartier divisor and proceed by induction. More precisely, let $H$ be an ample Cartier divisor on $X$ and write $\mathcal{I}$ for the ideal sheaf defining $C$. Then for $m \gg 0$, $\ox(mH)\otimes \mathcal{I}$ is globally generated. Thus we can find $D \sim_{\mathbb{Q}} mH$ such that $D$ contains $C$ but doesn't vanish at the generic point of any component of the fibre over $y$. Then no component of $D$ can be contracted over $y$, and thus there is some horizontal component $Z$ containing $C$. Replacing $X$ with $Z$ we see that the result holds by induction on $k$.
\end{proof}

\begin{proposition}\label{p-descent-over-curve}
Let $X$ be a normal integral scheme of dimension at most 3 and let $f \colon X \to C$ be a projective contraction over $S$ with generic fibre $X_{k(C)}$.
Suppose that $C$ is a regular 1-dimensional scheme and $L$ is a nef $\mathbb{Q}$-Cartier $\mathbb{Q}$-divisor on $X$ such that $L|_{X_{k(C)}} \equiv 0$.
Then there exists a $\mathbb{Q}$-Cartier $\mathbb{Q}$-divisor $D$ on $C$ such that $L \equiv f^*D$.
\end{proposition}

\begin{proof}
By \autoref{l-descent-num-triv-base}, $L \equiv_f 0$. 
If $C$ is not contracted over $S$ then in fact $L \equiv_{S}0 $. Otherwise $C$ is projective over a field and we conclude by the arguments of \cite[Lemma 5.2]{BW17}.
\end{proof}

\subsection{Generalised pairs}

Generalised pairs have been introduced in \cite{BZ16} and since then they revealed to be a powerful tools in birational geometry over fields of characteristic 0.
Here $\mathbb{K}$ is either $\mathbb{Z}$ or $\mathbb{Q}$.

\begin{definition}
For an integral normal scheme $X$, an integral $\mathbb{K}$-b-divisor is an element
$$\mathbf{D} \in \mathbf{WDiv}(X)_{\mathbb{K}}=\lim_{Y \to X} \WDiv(Y)_{\mathbb{K}},$$
where $Y \to X$ run through all possible proper birational morphisms of normal schemes.
Given $Y \rightarrow X$ a proper birational morphism of normal schemes, we have a natural map 
$\text{tr}_{Y} \colon \textbf{WDiv}(X)_{\mathbb{K}} \to \text{WDiv}(Y)_{\mathbb{K}},$
called the trace. 
 We denote $\Bf{D}_{Y}= \text{tr} (\Bf{D})_{Y}$.

We say that $\Bf{D}$ is a \emph{$\mathbb{K}$-b-Cartier $\mathbb{K}$-b-divisor} if there is a model $X'\to X$ such that $\Bf{D}_{X'}$ is $\mathbb{K}$-Cartier and for any $\phi\colon X'' \to X'$ the equality $\Bf{D}_{X''}=\phi^{*}\Bf{D}_{X'}$ holds. 
In this case we say that $\Bf{D}$ \emph{descends} to $X'$.	
\end{definition}

Every $\mathbb{K}$-Cartier $\mathbb{K}$-divisor $D$ on $X$ induces a natural $\mathbb{K}$-Cartier $\mathbb{K}$-b-divisor $\Bf{D}=\overline{D}$ as follows:
$$(\overline{D})_Y=f^*D \text{ where } f\colon Y \to X \text{ is a proper birational morphism}.$$
In particular, $\Bf{D}=\overline{\textbf{D}}_{X'}$ if and only if $\Bf{D}$ descends on $X'$. 

\begin{definition}
Let $X \to T$ be a proper morphism of quasi-projective normal schemes over $R$.
Let $\Bf{D}$ be a $\mathbb{K}$-Cartier $\mathbb{K}$-b-divisor on $X$ and let $X'$ be a model on which $\Bf{D}$ descends.
If $\Bf{D}_{X'}$ is a nef (resp.~ semiample) $\mathbb{K}$-Cartier $\mathbb{K}$-divisor over $T$, we say that $\Bf{D}$ is \emph{b-nef} (resp.~ \emph{b-semiample}) over $T$.
\end{definition}

Clearly if $D$ is nef (resp.~ semiample) then $\overline{D}$ is b-nef (resp.~ b-semiample).
\begin{definition}
	A \emph{generalised pair} (or $g$-pair) $(X,B+\Bf{M})$ over $T$ consists of 
	\begin{enumerate}
		\item a projective morphism $f \colon X \to T$ of normal quasi-projective $R$-schemes;
		\item an effective $\mathbb{Q}$-divisor $B$ (called the boundary part);
		\item a b-nef $\mathbb{Q}$-Cartier $\mathbb{Q}$-b-divisor $\Bf{M}$ (called the moduli part);
		\item $K_X+B+\Bf{M}_X$ is $\mathbb{Q}$-Cartier.
	\end{enumerate}
\end{definition}

 If $\Bf{D}$ is b-nef (resp b-semiample) then for any $X'$ with $\Bf{D}=\overline{\Bf{D}_{X'}}$, the divisor $\Bf{D}_{X'}$ is nef (resp.~ semiample).
We recall the definition of singularities for generalised pairs.

\begin{definition}
	Let $(X, B+\Bf{M})$ be a $g$-pair over $T$.
	For every  proper birational morphism $\pi \colon Y \to X$ of normal schemes we can write 
	$$K_Y+B_Y+\Bf{M}_{Y} = \pi^* (K_X+B+\Bf{M}_{X}). $$ 
	The \emph{generalised discrepancy} of $E$ is $a(E, X, B+\Bf{M}) = -\coeff_{E}(B_Y).$
	We say $(X,B+\Bf{M})$ is \emph{generalised klt} (resp.~ \emph{generalised lc})  if $a(E,X,B+M) > -1$ (resp.~ $a(E,X,B+M) \geq -1$) for all divisors $E$ appearing on some birational model.
	
	If $(X,B+\Bf{M})$ is generalised lc, $(X,B)$ is dlt and $(X,B+(1+t)\bf{M})$ is generalised lc for some $t >0$ then we say the pair is \emph{generalised dlt}.
\end{definition}
 
If $(X,B)$ is klt/lc/dlt and $N$ is a nef $\mathbb{K}$-Cartier $\mathbb{K}$-divisor then $(X,B+\bf{N})$ is always generalised klt/lc/dlt for $\bf{N}=\overline{N}$.

We will often use the following result on singularities of surfaces.

\begin{lemma}\label{l-Q-fact}
	Let $(X, \Delta)$ be a dlt surface pair. Then $X$ has rational and $\mathbb{Q}$-factorial singularities.
\end{lemma}

\begin{proof}
	By \cite[Proposition 2.28]{kk-singbook}, dlt surface singularities are rational and rational surface singularities are $\mathbb{Q}$-factorial by \cite[Proposition 10.9]{kk-singbook}.
\end{proof}

For surfaces, the MMP for generalised pairs is immediate from the usual MMP as the moduli part $\Bf{M}$ is nef on every model.

\begin{proposition}\label{g-MMP-surfaces}
	Let $(X,B+\Bf{M})$ be a generalised dlt projective surface over $T$.
	Then we can run a $(K_X+B+\Bf{M}_X)$-MMP over $T$ which terminates.
\end{proposition}

\begin{proof}
As $X$ is $\mathbb{Q}$-factorial, $\Bf{M}_X$ is a $\mathbb{Q}$-Cartier $\mathbb{Q}$-divisor.
We start by proving that $\Bf{M}_X$ is nef. 
If $\pi \colon Y \to X$ is a proper birational morphism on which $\Bf{M}_Y$ is nef, then by projection formula we have that for every curve $C \subset X$,
$$ \Bf{M}_X \cdot C =\pi_*\Bf{M}_Y \cdot C= \Bf{M}_Y \cdot \pi^* C \geq 0. $$
This shows $\Bf{M}_X$ is nef.
As $\Bf{M}_X$ is a nef, then a step of a $(K_X+B+\Bf{M}_X)$-MMP is thus also a step of a $(K_X+B)$-MMP, which we know exists and terminates by \cite{Tan18}. 
\end{proof}

\subsection{Birational geometry of klt Calabi--Yau pairs}

Let $S$ be a quasi-projective normal integral scheme over $R$.
We study birational properties of Calabi--Yau pairs over $S$.
	\begin{definition}
 	We say $(X,\Delta)$ is a \emph{klt Calabi--Yau pair} (abbreviated as klt CY pair) over $S$ if $(X, \Delta)$ is a klt pair, proper over $S$ such that $K_X+\Delta \equiv 0$ over $S$.	
 	\end{definition}
	If $S$ is clear from the context, we will simply omit it.
	We discuss some results on the birational geometry of klt CY pairs.
	
	\begin{lemma}\label{l-image-CY}
		Let $(X,\Delta)$ be a klt CY pair over $S$.
		Let $\pi \colon X \to Y$ be a proper birational contraction between normal proper $S$-schemes.
		Then $(Y, \pi_*\Delta)$ is crepant birational to $(X, \Delta)$. In particular, $(Y, \pi_*\Delta)$ is a klt CY pair over $S$.
	\end{lemma}
	
	\begin{proof}
		As $K_X+\Delta \equiv 0$, then clearly $K_Y+\pi_*\Delta \equiv 0$. Then by negativity lemma \cite[Lemma 2.14]{bhatt2020}, we conclude $K_X+\Delta \sim_{\mathbb{Q}} \pi^*(K_Y+\pi_*\Delta)$.
	\end{proof}
	
	\begin{lemma}\label{l-rat-CY}
		Let $k$ be an algebraically closed field.
		Let $X$ be a projective surface over $k$ such that $K_X \equiv 0$.
		Suppose that
		\begin{enumerate}
			\item if $k = \overline{\mathbb{F}}_p$, $X$ is klt;
			\item if $k \neq \overline{\mathbb{F}}_p$, $X$ is $\bQ$-factorial (e.g. $X$ is klt).
		\end{enumerate}
		If the singularities of $X$ are worse than canonical, then $X$ is birational to $\mathbb{P}^{2}_{k}$.
	\end{lemma}
	\begin{proof}
		Let $\pi \colon Y \to X$ be the minimal resolution. 
		We have $K_Y + E =\pi^*K_X=0,$
		where $E>0$.
		Suppose for contradiction that $Y$ is not a rational surface.
		Therefore there exists $f \colon Y \to B$ where $B$ is a curve of genus $g(B) \geq 1$.
		We now claim that all irreducible components of $E$ are rational curves.
		This is guaranteed by (a) and (b): the minimal resolution of a klt singularity is a tree of rational curves (see the classification in \cite[3.40]{kk-singbook}) and
		in the case where $k \neq \overline{\mathbb{F}}_p$ and $X$ is $\mathbb{Q}$-factorial we apply \cite[Theorem 3.20]{Tan14} to conclude.
		Therefore each irreducible components of $E$ must be contained in a fibre of $f$. 
		Let $F$ be a general fibre of $f$. As $E \cdot F=0$, by adjunction we know $(K_Y+E) \cdot F=K_Y \cdot F = \deg K_F=-2$, contradicting $K_Y+E \equiv 0$. 
	\end{proof}
	
We recall a straightforward application of the abundance theorem to the semiampleness conjecture.
We say that a $\mathbb{Q}$-Cartier $\mathbb{Q}$-divisor $D$ on $X$ is \emph{$\mathbb{Q}$-effective} if there exists $n>0$ such that $H^0(X, \mathcal{O}_X(nD)) \neq 0.$
	
	\begin{proposition}\label{l-abund-effective}
	Let $(X, \Delta)$ be a klt CY pair over $S$ and let $L$ be a nef $\mathbb{Q}$-Cartier $\mathbb{Q}$-divisor on $X$.
	Suppose that 
	\begin{enumerate}
		\item $\dim X=2$; or
		\item $\dim X=3$, the image of $X$ is positive dimensional and the residue field of any closed point is of characteristic $p>5$.
	\end{enumerate}
If $L$ is $\mathbb{Q}$-effective, then $L$ is semiample.
Moreover, if $L \equiv E$ over $S$ where $E$ is an effective $\mathbb{Q}$-divisor, then $L$ is num-semiample.

\end{proposition}
\begin{proof}
	By assumption, there exists an effective $\bQ$-Cartier $\bQ$-divisor $E$ such that $L \sim_{\bQ} E$.
	If we consider $0< \varepsilon \ll 1$, the pair $(X, \Delta +\varepsilon E)$ is klt and $K_X+\Delta+\varepsilon E \sim_{\bQ} \varepsilon E$.
	By the abundance theorem for klt surfaces and threefolds (\cite[Theorems 1.1, 3.1]{BBS}) the divisor $E$, and thus $L$, is semiample.
	The last assertion now follows immediately.
\end{proof}

\section{Semiampleness for klt Calabi--Yau surfaces}

In this section we prove the semiampleness conjecture for klt CY surface pairs. 
Our principal tools are the MMP for excellent surfaces \cite{Tan18}, the classification of Bombieri-Mumford of smooth varieties with trivial canonical class over an algebraically closed field \cite{BM76} and the abundance theorem for surfaces \cite{tanaka2020abundance}.
To treat the case of imperfect fields we use the base change formula of \cite{Tan18b, PW22}.
The main difficulty in the proof lies in ruling out the existence of strictly nef divisors which are not ample on klt CY surfaces.

\subsection{Semiampleness for canonical $K$-trivial surfaces}
We fix $k$ to be a field of characteristic $p>0$.
We prove the semiampleness conjecture for $K$-trivial projective surfaces with canonical singularities over algebraically closed $k$.
The result is probably well-known to the experts but we include a proof for sake of completeness. 
We start with the following lemma on abelian varieties. 

\begin{lemma}\label{l-abelian-var}
	Let $k$ be an algebraically closed field and let $A$ be an abelian variety over $k$.
	If $L$ is a strictly nef line bundle on $A$, then $L$ is ample.
\end{lemma}

\begin{proof}
	The proof of \cite[Proposition 1.4]{Ser95} works over any algebraically closed field of arbitrary characteristic.
\end{proof}

\begin{proposition}\label{l-not-strictly-nef}
	Let $k$ be an algebraically closed field of characteristic $p>0$ and let $X$ be a smooth projective surface over $k$ such that $K_X \equiv 0$.
	If $L$ is a strictly nef Cartier divisor on $X$, then $L$ is ample.
\end{proposition}

\begin{proof}
	By the Bombieri-Mumford classification (\cite[page 1]{BM76}), we have $\chi(X, \MO_X) \geq 0$.
	If $\chi(X, \MO_X)>0$,  then we conclude $h^0(X,L)>0$ by the Riemann-Roch theorem. Thus $L^2>0$ and so $L$ is ample by the Nakai-Moishezon criterion.
	
	If $\chi(X, \MO_X)=0$, then by \cite{BM76, BM77} $X$ is either an abelian surface, a hyperelliptic surface or a quasi-hyperelliptic surface.
	If $X$ is hyperelliptic, we consider a finite \'etale cover $f \colon Y \to X$ where $Y$ is an abelian variety.
	Then $f^*L$ is strictly nef and we conclude that $f^*L$, and thus $L$, is ample by \autoref{l-abelian-var}.
	If $X$ is quasi-hyperelliptic, by \cite[Theorem 1]{BM76} there exists a finite morphism $f \colon E \times \mathbb{P}^1 \to X$, where $E$ is an elliptic curve.
	Since $f^*L$ is strictly nef and $\overline{\text{NE}}(E \times \mathbb{P}^1)=\textup{NE}(E \times \mathbb{P}^1)$, we conclude by Kleimann's criterion (see \cite[Theorem 1.4.29]{Laz04}) that $f^*L$, and thus $L$, is ample.
\end{proof}

\begin{lemma}\label{l-C-trivial}
    Let $X$ be a normal projective surface over a field $k$. Suppose that $L$ is a nef line bundle with $\nu(L)=1$ and $C$ a curve with $L\cdot C=0$. 
    Then $C^{2}\leq 0$. 
    Moreover, $C^2=0$ if and only if $L\equiv tC$ for some $t>0$.
\end{lemma}
\begin{proof}
If $C^{2}>0$ then $C$ is big and we write $C\sim_{\mathbb{Q}} A+E$ for $A$ ample and $E \geq 0$. Then $L\cdot C=L \cdot (A+E)\geq L \cdot A > 0$ since $L$ is nef, and hence pseudo-effective, but not numerically trivial.
If $C^2=0$ then we conclude that $L\equiv tC$ by the Hodge index theorem. As $C$ is effective and $\nu(L)=1$, then $t>0$.
\end{proof}

\begin{proposition}\label{p-abundance-smoothK}
	Let $k$ be an algebraically closed field of characteristic $p>0$ and let $X$ be a projective surface over $k$ with canonical singularities such that $K_X \equiv 0$.
	If $L$ is a nef Cartier divisor, then it is num-semiample.
\end{proposition}
\begin{proof}
	By passing to the minimal resolution and the base point-free-theorem \cite[Proposition 2.1.a]{Ber21}, it is sufficient to discuss the case of smooth surfaces with numerically trivial canonical class.
	If $\nu(L)=0$, the claim is obvious and if $\nu(L)=2$ we conclude by the base-point-free theorem \cite[Theorem 4.2]{Tan18}.
	
	We now suppose that $\nu(L)=1$.
	As $L$ is not strictly nef by \autoref{l-not-strictly-nef}, then there exists a curve $C$ such that $L \cdot C =0$. By \autoref{l-C-trivial}, either $C^{2} =0$ and $L\equiv tC$ for some $t > 0$ or $C^{2} <0.$ 
	Take $\varepsilon >0$ with $(X,\varepsilon C)$ klt. If $C^{2}=0$ we conclude that $K_{X}+\varepsilon C\equiv \varepsilon C$ is semiample by \autoref{l-abund-effective}, and hence $L$ is num-semiample. Otherwise $C^{2} <0$ then we can contract $C$ as step of $(K_{X}+\varepsilon C)$-MMP which is $L$-trivial. 
	We thus reduce to Calabi-Yau surface with canonical singularities by \autoref{l-image-CY} with smaller Picard rank. After finitely many steps, as $L$ is not strictly nef by \autoref{l-not-strictly-nef}, there is a curve $C$ with $C^2=0$ and we conclude.
\end{proof}

\subsection{Klt CY surfaces}
We now prove the semiampleness conjecture for klt CY surface pairs $(X,\Delta)$ over an arbitrary field $k$.
We start by discussing the case where the boundary divisor $\Delta$ is empty.

\begin{proposition} \label{p-babycase}
	Let $k$ be a field of characteristic $p>0$ and let $X$ be a klt projective surface over $k$ such that $K_X \equiv 0$.
	Let $L$ be a nef Cartier divisor on $X$ with $\nu(L)=1$. Then $L$ is num-semiample.
\end{proposition}
\begin{proof}
    Without loss of generality we can suppose that $k=H^0(X, \mathcal{O}_X)$.
	We divide the proof in several steps.
	\setcounter{step}{0}
	\begin{step}
		We can suppose that for all irreducible curves $C\subset X$ such that $L \cdot C=0$, then $C^2=0$.
	\end{step}
	\begin{proof}
		We first note that $C^2 \leq 0$ by \autoref{l-C-trivial}.
		Let us consider a sufficiently small rational number $\varepsilon >0$ such that $(X, \varepsilon C)$ is klt.
		If $C^2<0$, by \cite[Theorem 4.4]{Tan18} there exists a birational morphism $\varphi \colon X \to Y$ such that $\text{Ex}(\varphi)=C$ and a Cartier divisor $L_Y$ such that $L \sim \varphi^* L_Y$.
		Thus it is sufficient to prove that $L_Y$ is num-semiample.
		Since $Y$ is a klt Calabi--Yau surface by \autoref{l-image-CY}, this process will terminate after a finite number of steps as the Picard number decreseas by 1 at each step.
	\end{proof}
	\begin{step}
		If there exists a curve $C$ such that $L \cdot C =0$, then $L$ is num-semiample.
	\end{step}
	\begin{proof}
	    By \autoref{l-C-trivial} we have $L \equiv tC$ for some $t >0$. However $(X,\varepsilon C)$ is klt for some $\varepsilon >0$ and so we conclude that $L$ is num-semiample by \autoref{l-abund-effective}.
	\end{proof}
	Suppose now $L$ is strictly nef. To conclude the proof it is sufficient to show that $L$ is ample.
	\begin{step}
		We can suppose that $X$ is geometrically normal.
	\end{step}
	\begin{proof}
		Suppose $X$ is not geometrically normal.
		Let $Y$ be the normalisation of $(X \times_k \overline{k})_{\text{red}}$ and let
		$f \colon Y \to X$ be the natural morphism. 
		Since $X$ is not geometrically normal and $k=H^0(X, \mathcal{O}_X)$ is algebraically closed in $k(X)$, by \cite[Theorem 1.1]{Tan18b} there exists an effective divisor $E>0$ such that
		\[K_Y+E \sim f^*K_X \equiv 0.\]
	     By Serre duality we have 
	     $H^2(Y, \mathcal{O}_Y(f^*mL)) \simeq H^0(Y, \mathcal{O}_Y(K_Y-f^*mL))^* ,$
	     which vanishes as $(K_Y-f^*mL) \cdot A = (-E-f^*mL) \cdot A<0$ for an ample Cartier divisor $A$ and $m > 0$.
	     Let $\pi \colon Z \to Y$ be the minimal resolution and write $K_Z+F = \pi^*K_Y,$ for some effective $\pi$-exceptional $\mathbb{Q}$-divisor $F$.
	    By the Riemann-Roch theorem \cite[Theorem 2.10]{Tan18} and the projection formula we deduce
	    \begin{equation*} \label{eq1}
        \begin{split}
        h^0(Z, \mathcal{O}_Z(m\pi^*f^*L)) & \geq \chi(Z, \MO_Z) +m\pi^*f^*L \cdot (m\pi^*f^*L-K_Z) \\
         & = \chi(Z, \MO_Z) + mf^*L\cdot E,
        \end{split}
\end{equation*}
where the last inequality follows from $L^2=0$, $\pi^*f^*L \cdot F=0$ and $K_Y \equiv -E$.
		Since $L$ is strictly nef, we have $f^*L \cdot E>0$ and thus $h^0(Y, \mathcal{O}_Y(mf^*L))>0$ for $m$ sufficiently large. As $f^*L$ is strictly nef and effective, then $f^*L$ is ample. Therefore $L$ is ample, concluding.
	\end{proof}
	
	We now prove the assertion of the theorem when $X$ is geometrically normal.
	Note that if $\overline{k} \simeq \overline{\mathbb{F}_p}$ (resp.~ $\overline{k} \neq \overline{\mathbb{F}_p}$), then $X_{\overline{k}}$ is klt (resp.~ $\mathbb{Q}$-factorial).
	Indeed, $X_{k^{\text{sep}}}$ is klt by \cite[Proposition 2.15]{kk-singbook} and thus $\mathbb{Q}$-factorial by \autoref{l-Q-fact}.
	Therefore by \cite[Lemma 2.5]{Tan18b} $X_{\overline{k}}$ is $\mathbb{Q}$-factorial. We now divide the proof in two cases according to the singularities of $X_{\overline{k}}$.
	If $X_{\overline{k}}$ has canonical singularities, we conclude by \autoref{p-abundance-smoothK}.
	If the singularities of $X_{\overline{k}}$ are worse than canonical, then \autoref{l-rat-CY} guarantees that $X_{\overline{k}}$ is a rational surface.
	Let us consider the minimal resolution $\pi \colon Y \to X_{\overline{k}}$.
	Since $Y$ is a smooth rational surface, $H^1(Y, \MO_Y)=0$ and thus by the Leray spectral sequence we deduce that $H^1(X_{\overline{k}}, \MO_{X_{\overline{k}}} )=0$.
	By flat base change, we conclude $H^1(X, \MO_X)=0$
	and thus $\chi(X, \MO_X) \geq 1$.
	Since $\nu(L) \geq 1$, we have $H^2(X, \mathcal{O}_X(L)) \simeq H^0(X, \mathcal{O}_X(K_X-L))^*=0$ and by Riemann-Roch we deduce
	$h^0(X, \mathcal{O}_X(L)) \geq \chi(X, \MO_X) \geq 1$.
	Therefore $L$ is ample, concluding the proof.
\end{proof}

\begin{theorem}\label{t-num-semiampleness}
	Let $k$ be a field and let $(X, \Delta)$ be a projective klt surface pair such that $K_X+ \Delta$ is $\mathbb{Q}$-Cartier and $K_X+\Delta \equiv 0$.
	If $L$ is a nef $\mathbb{Q}$-Cartier $\mathbb{Q}$-divisor on $X$, then $L$ is num-semiample.
\end{theorem}

\begin{proof}
    If the characteristic of $k$ is 0, this is \cite[Theorem 8.2]{LP20}. So we suppose the characteristic is $p>0$ and we subdivide the proof according to the numerical dimension of $L$.
	If $\nu(L)=0$, the claim is obvious. 
	If $\nu(L) =2$, then $L$ is big and nef and we conclude by the base-point-free theorem (\cite[Theorem 4.2]{Tan18}).
	
	The only case we thus need to study in detail is when
	$\nu(L)=1$. We can suppose $L$ is Cartier.
	By Serre duality, $h^2(X,\mathcal{O}_X(L))=h^0(X, \mathcal{O}_X(K_X -L))$ which vanishes as $\nu(L) \geq 1$.
	The strategy is to reduce to \autoref{p-babycase}.
	For this we subdivide the proof in various steps.
	\setcounter{step}{0}
	
	\begin{step}\label{s-LD=0}
		If $L \cdot \Delta>0$, then $L$ is semiample.
	\end{step}
	\begin{proof}
		Let $\pi \colon W \to X$ be the minimal resolution and let $K_W+\Delta_W = \pi^* (K_X+\Delta)$.
		By the Riemann-Roch theorem, we have
		\begin{align*}
			h^0(W, \mathcal{O}_W(m\pi^*L)) & \geq \chi(W, \MO_W) + m\pi^*L \cdot (m\pi^*L - K_W) \\
			&=\chi(W, \MO_W) + m\pi^*L \cdot \Delta_W= \chi(W, \MO_W)+mL \cdot \Delta.
		\end{align*}
		In particular, for $m \gg 0$ we have $h^0(X, \mathcal{O}_X(mL)) \neq 0$ and we conclude that $L$ is semiample by \autoref{l-abund-effective}.
	\end{proof}
	
	\begin{step}
		We can assume that every irreducible curve $C$ contained in the support of $\Delta$ satisfies $C^2 \leq 0$.
	\end{step} 
	\begin{proof}
		Suppose that there exists a curve $C$ in the support of $\Delta$ such that $C^2>0$.
		Then $\Delta=  a C + \Gamma$ and $K_X+\Gamma = -a C$ is a big and nef $\bQ$-Cartier $\bQ$-divisor. In particular, $(X, \Gamma)$ is klt and $L-(K_X +\Gamma)$ is a big and nef $\bQ$-Cartier $\bQ$-divisor. Thus we conclude $L$ is semiample by the base-point-free theorem.
	\end{proof}
	
	From now on, we assume that $L \cdot \Delta =0$ and all curves $C$ in the support of $\Delta$ satisfy $C^2 \leq 0$.
	
	\begin{step}
		We can suppose that each curve $C$ in $\Delta$ satisfies $C ^2=0$. In particular, $\Delta$ is a nef $\bQ$-Cartier $\bQ$-divisor.
	\end{step}
	\begin{proof}
		Suppose there exists a curve $C$ such that $\Delta=aC+\Gamma$, $C$ is not contained in the support of $\Gamma$ and $C^2<0$.
		Then we have 
		\[(K_X+C) \cdot C \leq (K_X + \Delta + (1-a)C) \cdot C= (1-a)C^2<0.\]
		Thus by \cite[Theorem 2.10]{tanaka2020abundance} there exists a birational map $\pi \colon X \rightarrow Y$ such that $\text{Ex}(\pi)=C$ and $L=\pi^*L_Y$ for some $\bQ$-Cartier $\bQ$-divisor $L_Y$ on $Y$.
		Moreover $(Y, \pi_*\Delta)$ is a klt Calabi--Yau pair by \autoref{l-image-CY}.
		After a finite number of such contractions, we end up with a klt Calabi--Yau pair $(Z, \Delta_Z)$ such that all irreducible curves $C$ in $\Delta_Z$ satisfy $C^2=0$. 
	\end{proof}
	
	\begin{step}
		We can suppose $\Delta=0$.
	\end{step}
	\begin{proof}
		If $\Delta \neq 0$, then $\Delta$ is an effective nef $\mathbb{Q}$-Cartier $\bQ$-divisor with numerical dimension $\nu(\Delta) \geq 1$. Since it is not big by Step \autoref{s-LD=0}, $\nu(\Delta) = 1$. 
		By \autoref{l-abund-effective} $\Delta$ is semiample, and we denote by $g \colon X \to B$ the induced contraction. 
		Since $L \cdot \Delta=0$, by \autoref{p-descent-over-curve} there exists an ample $\bQ$-Cartier $\bQ$-divisor $A$ on $B$ such that $L \equiv g^*A$.
	\end{proof}
	Therefore we reduced to prove the theorem in the case where $X$ is a projective surface with klt singularities and $K_X \equiv 0$, which we proved in \autoref{p-babycase}.
\end{proof}

\section{Generalised abundance for surfaces}

In \cite[Theorem B]{LP20} the authors show that,  over fields of characteristic 0, \autoref{c-gen-ab} is implied by the standard conjectures of the MMP and the semiampleness conjecture.
Their arguments heavily rely on the canonical bundle formula and therefore do not extend to positive or mixed characteristic.
Despite this obstacle, in this section we show the generalised abundance conjecture (and variants thereof) in the case of excellent surfaces.

\subsection{Surfaces over a field}

For excellent klt surfaces the base-point-free theorem (resp.~ abundance) has been proven in \cite[Theorem 4.2]{Tan18} (resp.~ in \cite{tanaka2020abundance} and \cite[Theorem 3.1]{BBS}).
These results together with the MMP and \autoref{t-num-semiampleness} are sufficient to prove \autoref{c-gen-ab} for surfaces over fields. Our strategy is similar to \cite{LP20}.

\begin{proposition}\label{T-gen-abund-surf}
Let $(X,B)$ be a projective klt surface pair over a field $k$. Suppose that $K_{X}+B$ is a pseudo-effective $\mathbb{Q}$-Cartier and let $M$ be a nef $\mathbb{Q}$-Cartier $\mathbb{Q}$-divisor such that $L:=K_X+B+M$ is nef.
Then $L$ is num-semiample.
\end{proposition}

\begin{proof}
Note first that if $L$ is big, then $2L-(K_{X}+B)=L+M$ is big and nef, hence $L$ is semiample by the base-point-free theorem (\cite[Theorem 4.2]{Tan18}).
	
As $X$ is a surface, by running a $(K_X+B)$-MMP and abundance (\cite{Tan18}) there is a birational contraction $g \colon X \to X_{\text{min}}$ together with a contraction morphism $h \colon X_{\text{min}} \to Z:=\Proj_k R(X, K_X+B)$ and an ample $\bQ$-Cartier $\bQ$-divisor $A$ on $Z$ such that $K_{X_{\text{min}}}+g_*B=h^*A$.
We denote by $f\colon X \to Z$ the composition $h \circ g$.
We divide the proof according to the dimension of $Z$.
\setcounter{cas}{0}
	\begin{cas}
	Suppose $\dim Z=2$. This means that $K_{X}+B$ is big and therefore also $L$, concluding.
	\end{cas}
	\begin{cas}
		Suppose $\dim Z=1$.
		If $L \cdot F=0$ for a general fibre $F$ of $f$, by \autoref{p-descent-over-curve} we have $L \equiv f^*N$ for some $\mathbb{Q}$-Cartier $\bQ$-divisor $N$ on $Z$. Note that $N$ must be an ample $\mathbb{Q}$-Cartier $\mathbb{Q}$-divisor, since if $C$ is a sufficiently general curve on $X$ we have $L \cdot C \geq (K_{X}+B) \cdot C > 0$ as $f_{*}C=Z$. 
		
		Suppose now that $L \cdot F > 0$, \emph{i.e.} $L$ is relatively big over $Z$. It is sufficient to show that $L$ big to conclude.
		By the negativity lemma we deduce $K_{X}+B=g^{*}(K_{X_{\text{min}}}+g_*B)+E$, where $E$ is an effective $\mathbb{Q}$-divisor contracted by $f$.
		In particular $K_X+B \sim_{\mathbb{Q}}f^*A+E$ where $A$ is an ample $\mathbb{Q}$-Cartier $\bQ$-divisor on $Z$. 
		As $L=(f^*A+E)+M$ is relatively big and $f^*A+E$ is numerical trivial on the generic fibre, then $M$ is relatively big.
		Note that $f^*A+M$ is nef and $(f^* A+M)^2 \geq 2 f^*A \cdot M >0$ and thus big.
		Therefore $L$ is big, concluding.
	\end{cas}
	
	\begin{cas}
		Finally suppose that $\dim Z = 0 $. Let $\pi\colon X \to Y$ be the minimal model of a $(K_{X}+B)$-MMP. Then $K_{Y}+B_{Y} \sim_{\mathbb{Q}} 0$ and $L'=\pi_{*}L \sim_{\mathbb{Q}} \pi_*M$ is nef on $Y$, and hence num-semiample by \autoref{t-num-semiampleness}.
		Choose $0<t<1$ sufficiently small such that $X \to Y$ is the end-product of a $(K_{X}+B+tM)$-MMP. Then by the negativity lemma $K_{X}+B+tM\equiv t\pi^{*}L'+E$ for $E \geq 0$. 
		If $L'$ is big, then so too is $K_{X}+B+tM$ and also $L$, and the result follows as above. 
		
		Suppose now $\nu(L')=1$.
		As $L'$ is num-semiample, there exists a contraction $g \colon Y \to W$ with $\dim W=1$ such that $L'\equiv g^*A$ for an ample $\mathbb{Q}$-Cartier $\mathbb{Q}$-divisor $A$ on $W$. 
		Denote by $f \colon X \to W$ the composition $g \circ \pi$.
		As $L|_{X_{k(W)}} \equiv \pi^*L'|_{{X_{k(W)}} }\equiv 0$, we conclude that $L \equiv f^*D$ for some $D$ on $W$ by \autoref{p-descent-over-curve}. As $\nu(L')=1$, we conclude that $D$ is ample and thus $L$ is num-semiample.
		
		Finally, if $L'$ is numerically trivial, then
		$L \sim_{\mathbb{Q}} \pi^*L'+E \equiv E$, where $E$ is $\pi$-exceptional. As $L$ is nef, then $E \leq 0$ by the negativity lemma \cite[Lemma 2.14]{bhatt2020}. As $L \equiv E$ is nef, we conclude $E=0$. 	
	\end{cas}
\end{proof}

\subsection{Excellent case} 

We are now ready to prove \autoref{c-gen-ab} for projective surfaces over $R$.

\begin{theorem}\label{c-abundance-general-exc-surf}
	Let $\pi \colon X\to T$ be a projective $R$-morphism of quasi-projective integral normal schemes over $R$.
	Suppose that $(X,B)$ is a klt surface such that
	\begin{enumerate}
		\item $K_{X}+B$ is pseudo-effective over $T$;
		\item $M$ is a nef $\mathbb{Q}$-Cartier $\mathbb{Q}$-divisor over $T$;
		\item $L:=K_X+B+M$ is nef $\mathbb{Q}$-Cartier $\bQ$-divisor over $T$.
	\end{enumerate} 
	Then $L$ is num-semiample over $T$.
\end{theorem}

\begin{proof}
	Without loss of generality we can suppose $\pi$ is a surjective contraction between normal schemes.
	Suppose first $L$ is big, in which case $2L-(K_{X}+B)$ is big and nef, so $L$ is semiample by \cite[Theorem 4.2]{Tan18}. In particular, if $\dim(T)=2$ we conclude. From now on we suppose that $L$ is not big over $T$, or equivalently that $L|_{X_k(T)}$ is not big.
	
	If $\dim(T)=1$, then $X_{k(T)}$ is a curve. As $L|_{X_{k(T)}}$ is not big, then $L|_{X_{k(T)}} \equiv 0$ and we conclude by \autoref{p-descent-over-curve} that $L \equiv \pi^*A$ where $A$ is a nef $\mathbb{Q}$-divisor on $T$. As nef divisors on a curve are num-semiample, we conclude.

	If $\dim(T)=0$, we apply \autoref{T-gen-abund-surf}.
\end{proof}
In particular, the semiampleness conjecture holds for klt Calabi--Yau excellent surfaces pairs.
\begin{remark}
	Note that the klt assumption is necessary in \autoref{c-num-semiample} and \autoref{c-gen-ab} as there are counterexamples if $(X,B)$ is allowed to be lc (or even log smooth), even if $L$ is supposed to be big \cite[Examples 7.1, 7.2]{MNW15}.
\end{remark}

\begin{remark}
	 The semiampleness conjecture and generalised abundance are false for $\mathbb{R}$-divisors as shown by the examples on projective K3 surfaces constructed in \cite[Theorem 1.5.b]{FT18}.
\end{remark}

The pseudo-effectivity of $K_{X}+B$ is necessary as shown in \cite[Example 6.2]{LP20}.
In fact we can completely characterise when generalised abundance holds in the case of surfaces (cf. \cite[Theorem 3.13]{HL20}).

\begin{proposition}\label{c-failure-gen-ab}
	Let $(X,B+\Bf{M})$ be a generalised klt surface over $T$ and suppose that $\Bf{M}_X$ and $L=K_{X}+B+\Bf{M}_{X}$ are nef $\mathbb{Q}$-Cartier $\bQ$-divisors over $T$.
	If $L$ is not num-semiample, then
	\begin{enumerate}
	    \item $T$ is the spectrum of a field $k$;
	    \item $K_X+B$ is not pseudo-effective;
	    \item $-(K_X+B) \equiv t \Bf{M}_X$ for some $0<t \leq 1$;
	    \item $\nu(\Bf{M}_X)=1$.
	\end{enumerate}
\end{proposition}
    
\begin{proof}

   We first suppose $\dim(T) \geq 1$.  
   If $L$ is big, then it is num-semiample by the base-point-free theorem. Therefore the only interesting case is $L$ not big and $\dim(T)=1$. Then $L|_{X_{_{k(T)}}} \equiv 0$ and therefore $L \equiv g^*M$ by \autoref{p-descent-over-curve}, proving (a).
    
    We can thus suppose that $\dim(T)=0$, i.e. it is the spectrum of a field. 
    We first note that $K_X+B$ is not pseudoeffective by \autoref{c-abundance-general-exc-surf}.
    We now follow some of the ideas of \cite[Theorem 3.13]{HL20}.
    If $K_X+B+2\Bf{M}_X$ is big, then $2L-(K_X+B)$ is big and nef and thus $L$ is semiample by the base-point-free theorem.
    
    If not, then $(K_X+B+2\Bf{M}_X)^2=0$, which implies that $(K_X+B+\Bf{M}_X) \cdot \Bf{M}_X=0$ and $\Bf{M}_X^2=0$.
    By an application of the Hodge index theorem (see \cite[Lemma 3.2]{HL20}), we conclude there exists $a \in \mathbb{Q}_{\geq 0}$ such that $a\Bf{M}_X \equiv K_X+B+\Bf{M}_X$.
    Suppose $a \geq 1$. Then $K_X+B \equiv (a-1) \mathbf{M}_X$ is pseudo-effective and thus $L$ is num-semiample by \autoref{T-gen-abund-surf}, contradicting the hypothesis. 
    Therefore $0 \leq a<1$, which implies (c) as $-(K_X+B) \equiv (1-a) \mathbf{M}_X$.
    
    We are only left to check that $\nu(\mathbf{M}_X)=1$.
    As $\mathbf{M}_X$ is not big, then $\nu(\mathbf{M}_X)<2$. If $\nu(\Bf{M}_X)=0$, then $K_X+B$ is pseudo-effective, and thus $L$ is num-semiample by \autoref{T-gen-abund-surf}, contradiction.
\end{proof}

\subsection{Serrano's conjecture and numerical non-vanishing}

As a consequence we obtain a version of Serrano's conjecture for excellent klt surface pairs (see \cite[Corollary 1.8]{HL20} for a similar result in characteristic 0). We note that the case where the base is an imperfect field needs a careful analysis.

\begin{theorem}\label{Serrano}
Let $\pi \colon X \to T$ be a projective morphism of quasi-projective integral normal schemes over $R$
and suppose $(X,B)$ is a klt surface pair. 
If $M$ is a strictly nef Cartier divisor, then $L_{t}:=K_{X}+B+tM$ is ample for $t > 4$.
\end{theorem}

\begin{proof}
We can suppose $\pi$ is a contraction, by taking the Stein factorisation.
By the cone theorem for excellent surfaces \cite[Theorem 2.40]{bhatt2020}, $L_{t}$ is strictly nef for $t >4$. 
By \autoref{c-failure-gen-ab} we see that either $L_{t}$ is num-semiample or $L_{t}\equiv -s(K_{X}+B)$ for some $s>0$. 
In the first case $L_{t}$ is necessarily ample.

In the latter case, $H^0(X, \mathcal{O}_X)=k$ is a field, $T=\Spec(k)$ and we are left to show that the strictly nef $\mathbb{Q}$-divisor $L:=-(K_{X}+B)$ is ample. 
Suppose for contradiction it is not: in particular $L^2=0$. 
It is sufficient to show that $L$ is $\mathbb{Q}$-effective. 
We first show that $H^2(\mathcal{O}_X(nL))=0$ for $n$ sufficiently large.
Let $n>0$ such that $nL$ is Cartier, so that by Serre duality there is an isomorphism $H^2(X, \mathcal{O}_X(nL)) \simeq H^0(X, \mathcal{O}_X(K_X-nL))^*$.
If $H^0(X, \mathcal{O}_X(K_X-nL)) \neq 0$, for an ample Cartier divisor $H$ we have $0<(K_X-nL)\cdot H \leq K_X\cdot H -n$, which gives a contradiction for $n$ sufficiently large.

Let $f \colon Y:=(X \times_k \overline{k})_{\red}^{n} \to X$ be the normalisation of the base change to the algebraic closure and let $E \geq 0$ be the $\mathbb{Z}$-divisor for which $K_Y+E = f^*K_X$.
First suppose $E+f^{*} B > 0$, which implies $$f^* L \cdot (-K_Y)=f^*L \cdot (E-f^*K_X)=f^*L \cdot (E+f^*B) >0,$$
where we used the condition $L^2=0$ in the last equality.
Let $\pi \colon Z \to Y$ be the minimal resolution. Thus by Riemann-Roch formula we have
$$ h^0(Z, \mathcal{O}_Z(\pi^*f^*nL)) \geq \chi(Z,\mathcal{O}_Z)-\frac{1}{2}nf^*L \cdot K_Y >0$$
for sufficiently large $n>0$ for which $nL$ is Cartier. 
In particular, $f^*nL$ is effective, and thus ample. 
In particular, $L$ is $\mathbb{Q}$-effective and thus ample.

If $E +f^*B=0$, then $B=0$ and $X$ is geometrically normal by \cite[Theorem 1.1]{Tan18b}. 
We run a $K_X$-MMP which ends with a Mori fibre space $\pi \colon Y \to C$ and $-K_Y$ is strictly nef.
If $\dim(C)=0$, then $Y$ is a klt del Pezzo surface.
If $\dim(C)=1$, then $\rho(Y)=2$  and thus the cone theorem \cite[Theorem 2.40]{bhatt2020} implies that $Y$ is a klt del Pezzo surface.
In both cases, $Y$ is a geometrically normal del Pezzo surface and by \cite{Sch01} we know $H^1(Y, \mathcal{O}_Y)=0$.
As $X$ and $Y$ have both rational singularities by \autoref{l-Q-fact} we deduce that $H^1(X, \mathcal{O}_X)=0$. By Riemann-Roch we thus have
$$ h^0(X, \mathcal{O}_X(nL)) \geq \chi(X,\mathcal{O}_X)+\frac{1}{2}nL \cdot (nL-K_X) \geq 1,$$
which shows $L$ is ample.
\end{proof}

We can generalise this result immediately to threefolds in the setting of \cite{bhatt2020}.

\begin{corollary}\label{Serrano-3}
   Let $\pi \colon X \to T$ be a projective morphism of quasi-projective integral normal schemes over $R$ and $(X,B)$ be a klt threefold pair.
   Suppose that the closed points of $R$ have residue fields of characteristic $0$ or $p>5$. 	
    Let $M$ be a strictly nef Cartier divisor. 
    If $\dim(\pi(X)) \geq 1$, then $L_{t}:=K_{X}+B+tM$ is ample for $t > 4$.
\end{corollary}

\begin{proof}
We can suppose $\pi$ to be a contraction.
Let $F$ be the generic fibre of $X \to T$. If $F$ is a surface, then $L_{t}|_{F}$ is ample by \autoref{Serrano}. If instead $F$ is a curve then $L_{t}|_{F}$ is ample since it has positive degree. In particular $L_{t}$ is big.

Then by \cite[Theorem H]{bhatt2020}, $L_{t}$ is strictly nef and moreover $2L_{t}-(K_{X}+B)=M+L_{t}$ is nef and big. Thus we conclude $L_{t}$ is semiample by \cite[Theorem G]{bhatt2020}, and hence it is ample as claimed.
\end{proof}


We can also show the numerical non-vanishing conjecture for generalised surface pairs.
The strategy of the proof is similar to \cite{HL20}. However due to some complications over imperfect field, we apply Serrano's conjecture to solve the non-pseudo-effective case, instead of referring to the explicit classification of \cite[Example 1.1]{Sho00}.

\begin{theorem}[Numerical non-vanishing for generalised klt surfaces]\label{t-num-non-van}
If $(X,B+\Bf{M})$ is a generalised klt surface pair over $T$ and $K_X+B+\Bf{M}_X$ is pseudo-effective, then $K_X+B+\Bf{M}_X$ is num-effective.
\end{theorem}

\begin{proof}
Let $r \in \mathbb{Q}_{>0}$ such that $L:= r(K_X+B+\Bf{M}_X)$ is Cartier.
By running a $(K_X+B+\Bf{M}_X)$-MMP over $T$ we reduce to prove the statement in the case $K_X+B+\Bf{M}_X$ is nef.
By \autoref{c-failure-gen-ab} we are only left to prove the case when $-(K_X+B)$ is nef with $\nu(-(K_X+B))=1$ and $-(K_X+B) \equiv t \Bf{M}_X$.

By Serrano's conjecture \autoref{Serrano}, $-(K_X+B)$ is not strictly nef. Thus there exists an irreducible curve $C \subset X$ such that $-(K_X+B) \cdot C =0$.
By \autoref{l-C-trivial} we have $C^{2} \leq 0$ with equality if and only if $(-K_{X}+B)\equiv \lambda C$ for $\lambda >0$. Thus we may suppose that $C^{2} <0$. Hence we can contract it $X \to X'$ by running a $(K_X+B+\varepsilon C)$-MMP for $\varepsilon>0$ sufficiently small. This MMP is clearly $(K_X+B)$-trivial. Therefore $(X',B')$ is a klt pair with $-(K_{X'}+B')^2=0$ and $\nu(-(K_{X'}+B'))=1$ and we can repeat then the same procedure. After a finite number of steps this process must terminate. Replacing $(X,B)$ with the output either $L$ is strictly nef or we find a curve $C$ such that $-(K_X+B) \cdot C=0$ and $C^2=0$. In either case we conclude.
\end{proof}

\subsection{Semiampleness for lc pairs}

 In this section we study the generalised abundance for generalised lc  surface pairs under the assumption the b-nef part is b-semiample. 
 In characteristic 0, this is immediate from the Bertini theorem and the abundance for log canonical surfaces.
 We will overcome the lack of Bertini in positive and mixed characteristic via Keel-Witaszek's base-point-free theorem.
 For a nef line bundle $L$ on $X$ over $S$, we define the exceptional locus $\mathbb{E}(L)$ to be the union of all closed integral subschemes $Z \subset X$ such that $L|_Z$ is not relatively big over $S$.
 We recall the following semiampleness criterion for line bundles.

\begin{theorem}[\cite{witaszek2020keels}] \label{wit-keel}
Let $L$ be a nef line bundle on a scheme $X$
projective over an excellent Noetherian base scheme $S$. Then $L$ is semiample over $S$ if
and only if both $L|_{\mathbb{E}(L)}$ and $L|_{X_{\mathbb{Q}}}$ are semiample.
\end{theorem}

We first need an adjunction result.

\begin{lemma}\label{adj-sec}
    Let $X \to T$ be a projective contraction of integral, excellent, normal quasi-projective schemes over $R$. 
    If $X$ is a surface and $C$ is an irreducible curve on $X$ contracted over $T$. Suppose that $(X,C)$ is an lc surface and $C$ is an irreducible curve over $T$ such that $(K_{X}+C) \cdot C=0$. Then $(K_{X}+C)|_{C} \sim_{\mathbb{Q}} 0$.
\end{lemma}

\begin{proof}
As $C$ is contracted over $T$, it is a curve over some field $k$. 
The same proof of \cite[Theorem 2.13]{tanaka2020abundance} assures that $(K_{X}+C)|_{C} \sim_{\mathbb{Q}} 0$ as claimed. 
\end{proof}

The following theorem should be well known if $T$ is a scheme over $\Spec(\mathbb{Q})$. 
We include a proof as we lack a suitable reference and some arguments do not clearly carry over to the case that $T$ is not of finite type over a field.

\begin{theorem}\label{lc-abund}
    Let $\pi \colon X \to T$ be a projective $R$-morphism of quasi-projective integral normal schemes over $R$.
    Suppose that $(X,B+\Bf{M})$ is generalised lc surface pair over $T$ such that $\dim T >0$, $\Bf{M}$ is b-semiample and $L=K_{X}+B+\Bf{M}_{X}$ nef.
    Then $L$ is semiample.
\end{theorem}

\begin{proof}
By taking Stein factorisation we may assume $\pi$ is a contraction. If $\dim(T)=1$ and $L|_{X_{k(Z)}} \sim_{\mathbb{Q}} 0$, then we conclude by \autoref{p-descent-equidimensional}. 

We may assume that $L$ is big. After taking a dlt modification we may suppose that $(X,B)$ is a $\mathbb{Q}$-factorial dlt pair. 
Let $C$ be an integral curve on $X$ such that $L \cdot C=0$.
As $L$ is big, we deduce that $C^2<0$.
We write $B=B_{C}+t_{C}C$ for some $t_{C} \in [0,1]$ such that the support $B_C$ does not contain $C$. 
As $C^2 <0$, we deduce $$(K_{X}+C) \cdot C \leq (K_X+t_{C}C) \cdot C \leq (K_{X}+B) \cdot C \leq L \cdot C=0$$ since $\Bf{M}_{X}$ is nef.
Moreover this is a chain of equalities if and only if $t_{C}=1$ and $B_{C} \cdot C=\Bf{M}_{X} \cdot C=0$. 

If $(K_X+C) \cdot C <0$, by \cite[Theorem 4.4]{Tan18b} there exists a contraction $X \to Y$ such that its exceptional locus is $C$ and $Y$ is $\mathbb{Q}$-factorial. In this way we contract all the  curves satisfying  $(K_X+C) \cdot C <0$ and $C^2<0$. 
As this contracts only curves $C$ with $(K_{X}+B) \cdot C \leq 0$ so the pair $(X,B)$ remains lc. 
Moreover if $(K_{X}+B) \cdot C \leq 0$ then either $t_{C}=1$ and $(X,C)$ is lc or $t_{C}<1$ and $(X,t_{C}C)$ is klt, so $X$ remains klt and $\mathbb{Q}$-factorial by \autoref{l-Q-fact}. 

We may suppose that every curve $C_i$ for which $L\cdot C_i=0$ satisfies $C_i^2<0$ as $L$ is assumed to be big and the following equalities hold: 
$$(K_X+C_i) \cdot C_i=\Bf{M}_{X}\cdot C_{i}=B_{C_i} \cdot C_{i}=0$$
We  can write $B=\Delta+\sum_i C_{i}$ where $C_{i}$ are the curves with $L \cdot C_{i}=0$.
As observed above, $\text{Supp}(\Delta)$ is disjoint from $C:=\sum C_i$ and $C_i \cap C_j = 0$ if $i \neq j$.
Recall by construction that we have $\mathbb{E}[L] = C$ and $\Bf{M}_{X} \cdot C=0$. 
Since $\Bf{M}_X$ is b-semiample, there is a model $\pi \colon Y \to X$ with $\Bf{M}_{Y}$ semiample on $Y$ and $\pi_{*}\Bf{M}_{Y}=\Bf{M}_{X}$.

Let $D \sim_{\mathbb{Q}} \Bf{M}_{Y}$ be an effective section which does not contain any irreducible component of the strict transform of $C$. 
Then $\pi_*D \sim_\mathbb{Q} \Bf{M}_{X}$ does not contain any irreducible component of $C$. 
As  $\Bf{M}_X \cdot C =0$ we deduce $\Bf{M}_{X}|_C \sim_{\mathbb{Q}} 0$ and in particular $L|_{C}\sim_{\mathbb{Q}}(K_{X}+C)|_{C}$.
We then have $L|_{C} \sim_{\mathbb{Q}} 0$ by \autoref{adj-sec}.

The final argument differs depending on the characteristic. We may assume $T$ is local, and suppose first that the closed point has characteristic $p> 0$.
For dimension reasons, $L|_{X_{K(Z)}}$ is ample, therefore 
by \autoref{wit-keel} $L$ is semiample if and only if $L|_{\mathbb{E}[L]}$ is so. 
As $\mathbb{E}[L]$ is the disjoint union of curves $C_{i}$ with $L \cdot C_{i}=0$ and $L|_{C_{i}} \sim_{\mathbb{Q}} 0$ for every $i$ from above, we conclude.

Now suppose the closed point has residue field of characteristic $0$ instead. Take $k \in \mathbb{N}$ such that $L_{k}=kL-C$ is big, and in particular effective. Then $L_{k}$ intersects only finitely many curves negatively, call them $\gamma_{j}\subseteq \SB(L_{k})$. By construction $C \cdot C_{i} <0$ for each $i$, since $C$ is a disjoint union of irreducible components and every component has negative self-intersection. Hence $L \cdot \gamma_{j} > 0$ for each $j$, so increasing $k$ we may assume $L_{k}$ is strictly nef. Since it is big, $L_{k}$ is ample. In particular the stable base locus has $\SB(L) \subseteq C$. 

Now we have a short exact sequence
\[0 \to \ox((k+1)L-C) \to \ox((k+1)L) \to \ox[C]((k+1)L|_{C}) \to 0\]
where $(k+1)L-C=K_{X}+B+\textbf{M}_X+L_{k}$. Then $H^{i}(X, \mathcal{O}_X((k+1)L-C))=0$ by Kawamata Viewhweg vanishing \cite[Theorem 3.3]{Tan18}, since $L_{k}$ is ample, $(X,B)$ is lc and $X$ is klt. 
As $L|_C \sim_{\bQ} 0$,  we have $(k+1)L|_{C} \sim 0$ for $(k+1)$ sufficiently large and divisible. 
As $H^0(X, \mathcal{O}_X((k+1)L)) \to H^0(C, \mathcal{O}_C((k+1)L))$ is surjective, we conclude $\textbf{SB}(L)$ is empty, concluding.
\end{proof}


We prove a variant of the semiampleness conjecture for lc pairs.

\begin{corollary}\label{SQC-surf}
    Let $\pi \colon X \to T$ be a projective $R$-morphism of quasi-projective integral normal schemes over $R$.	
	Let $(X,B+\Bf{M})$ be a generalised lc surface over $T$ with $\Bf{M}$ b-semiample. 
	If $K_{X}+B+\Bf{M}_X$ is nef over $T$, then it is semiample over $T$. 
\end{corollary}

\begin{proof}
We can suppose $\pi$ to be surjective.
Otherwise, if $\dim T >0$ then this is \autoref{lc-abund}. 
Suppose then that $\dim T=0$, so it is the spectrum of a field $k$. 
If the characteristic of $k$ is 0, then we conclude by Bertini theorem and \cite[Theorem 6.1]{Fuj12}.
If $k$ is finite, we consider the base change to $\overline{k}$. Then $(X_{\overline{k}}, B_{\overline{k}}+M_{\overline{k}})$ is a generalised lc pair and, since $k \to \overline{k}$ is faithfully flat it is sufficient to prove $L_{\overline{k}}$ is semiample.
Suppose now that $k$ is an infinite $F$-finite field of characteristic $p>0$.
Then there exists an effective $\mathbb{Q}$-divisor such that $D \sim_{\mathbb{Q}} \Bf{M}$ and $(X,B+D)$ is lc by \cite[Theorem 1]{tanaka2017semiample}. 
Hence $L \sim_{\mathbb{Q}} K_{X}+B+D$ is semiample by \cite[Theorem 1]{tanaka2020abundance}.
		
If $k$ is not $F$-finite, then by standard arguments (cf. \cite{DW19}) there is an $F$-finite sub-field $l\subseteq k$ and a generalised klt pair $(X_{l},B_{l}+\Bf{M}_{X_{l}})$ such that
\begin{itemize}
	\item $(X,B+M)=(X_{\times_{\Spec(l)}} \Spec(k),B_{l}\times_{\Spec(l)} \Spec(k)+\Bf{M}_{X_{l}} \times_{\Spec(l)} \Spec(k))$; and
		\item $\Bf{M}_{X_{l}}$ is b-semiample.
	\end{itemize}
As $K_{X_l}+B_l+\Bf{M}_{X_{l}}$ is semiample, so is $K_{X}+B+\Bf{M}_{X}$.
\end{proof}

\section{Semiampleness for CY threefolds} \label{s-3folds}

In this section we show the semiampleness conjecture for threefolds in mixed characteristic (or over positive dimensional bases of positive characteristic).
We fix $R$ to be an excellent DVR with maximal ideal $\mathfrak{m}$, residue field $k:=R/\mathfrak{m}$ of characteristic $p>0$ and fraction field $K$.
We recall the following well-known fact on divisor class groups of projective morphisms over DVR.

\begin{lemma}\label{l-Picard-DVR} 
    Let $\pi \colon X \to \Spec(R)$ be a projective contraction of integral normal schemes.
    We write $(X_{k})_{\text{red}} = \sum_i D_i$ as a Weil divisor. 
    Then the following sequence is exact:
    $$\bigoplus_i \mathbb{Q}[D_i] \xrightarrow{j} \Cl(X)_{\mathbb{Q}} \to \Cl(X_{K})_{\mathbb{Q}} \to 0.$$
    Define $\Pic_{X_{k}}(X)_\mathbb{Q}:=j\left(\bigoplus_i \mathbb{Q}[D_i] \cap \Pic(X)_{\mathbb{Q}}\right)$  and denote by $N^{1}_{X_{k}}(X)_{\mathbb{Q}}$ its quotient by the numerical equivalence relation. Then the following sequences are exact: $$\Pic_{X_{k}}(X)_\mathbb{Q}\to \Pic(X/R)_{\mathbb{Q}} \to \Pic(X_K)_{\mathbb{Q}}, $$
    $$ N^{1}_{X_{k}}(X)_{\mathbb{Q}}\to \NS(X/R)_{\mathbb{Q}} \to \NS(X_{K})_{\mathbb{Q}}. $$
    If $X$ is $\mathbb{Q}$-factorial, then the second and third sequences are also surjective on the right.
\end{lemma}

\begin{proof}
    The first sequence is obtained from \cite[Proposition 6.5]{Ha77} by tensoring with $\mathbb{Q}$. 
    The second exact sequence follows immediately from the first one by considering the natural injections $\Pic(X/R) \to \text{Cl}(X)$ and $\Pic(X_{K}) \to \text{Cl}(X_{K})$. 
    The third exact sequence follows from the second and the fact that $L \equiv L'$ over $R$ if and only if $L|_{X_K} \equiv L'|_{X_K}$ by \autoref{p-descent-over-curve}.
    If $X$ is $\mathbb{Q}$-factorial then we have a natural isomorphism $\Pic(X_{K}) \simeq \text{Cl}(X)_{\mathbb{Q}}$, so the second sequence is surjective on the right. 
    Then the third sequence must also be surjective on the right as $\NS(X/R)_{\mathbb{Q}}$ is a quotient of $\Pic(X/R)_{\mathbb{Q}}$.
\end{proof}
    
We now prove the semiampleness conjecture for CY
threefolds over a DVR of residue characteristic $p>5$. Using the semiampleness conjecture for surfaces, we show a numerical non-vanishing on the threefold fibration. This allows to conclude by the abundance theorem for klt threefolds \cite{BBS}.

\begin{theorem}\label{t-CY-3folds}
    Let $R$ be an excellent DVR with residue characteristic $p>5$ and let $\pi \colon (X,B) \to \Spec(R)$ a projective klt CY pair of dimension 3. 
    If $\pi$ is surjective and $L$ is a nef $\mathbb{Q}$-Cartier $\mathbb{Q}$-divisor, then $L$ is num-semiample.
\end{theorem}

\begin{proof}
By taking a $\mathbb{Q}$-factorialisation \cite{BBS}, we can suppose $X$ to be $\mathbb{Q}$-factorial. 
If $\nu(L|_{X_K})=2$, then we conclude by the base-point-free theorem 
\cite[Theorem 9.15]{bhatt2020}. 
If $\nu(L|_{X_K})=0$, then $L|_{X_K} \equiv 0$ and  by \autoref{p-descent-over-curve} there exists a $\mathbb{Q}$-Cartier $\bQ$-divisor $D$ on $\Spec(R)$ such that $L \equiv \pi^*D$, concluding.

Suppose $\nu(L|_{X_K})=1$. Thanks to \autoref{l-abund-effective} it is sufficient to show that there exists an effective $\mathbb{Q}$-divisor $E$ such $L \equiv_{\pi} E$ .
By \autoref{t-num-semiampleness}, there exists $\varphi \colon X_K \to W$ such that $L_K \equiv \varphi^*A$ for some $A$ ample $\mathbb{Q}$-Cartier $\bQ$-divisor on $W$.
Take a Nagata compactification $\overline{W}$ of $W$ over $\Spec(R)$ (\cite[\href{https://stacks.math\cdot Columbia.edu/tag/0F41}{Tag 0F41}]{stacks-project}) and let $g \colon Z \to \overline{W}$ be a projective resolution of singularities \cite[Proposition 2.12]{bhatt2020}.
Resolving the indeterminacy of $t \colon X \dashrightarrow Z$ and applying again \cite[Proposition 2.12]{bhatt2020} there exist projective birational morphisms $\pi \colon Y \to X$ and $\psi \colon Y \to Z$ of normal projective $R$-schemes.
	
Write $(Y_k)_{\text{red}}=\sum_i E_i$ and $(Z_k)_{\text{red}}=\sum_j D_j$ as Weil divisors. 
By \autoref{l-Picard-DVR} we have the following commutative diagram of exact sequences:
	\[
	\xymatrix{
	\bigoplus \mathbb{Q}[D_j]	\ar[d]^{\psi^*} \ar[r] & \NS(Z/R)_{\mathbb{Q}} \ar[r] \ar[d]^{\psi^*} \ar[r] & \NS(Z_K)_{\mathbb{Q}} \ar[r] \ar[d]^{\psi_K^*} & 0 \\
     \NS^{1}_{X_{k}}(X)_{\mathbb{Q}} \ar[r] & \NS(Y/R)_{\mathbb{Q}} \ar[r] & \NS(Y_K)_{\mathbb{Q}}, &
	}
	\]
where the surjectivity of the top row comes from $Z$ being regular.
If we write $N:=\pi^*L,$ we know that $N_K = \pi|_K^*\varphi^* A $. 
Let $H := g_K^* \varphi^* A$ be the pull-back of $A$ via the morphism $Z_K \to W$.
By the above sequence there exists a $\mathbb{Q}$-Cartier $\mathbb{Q}$-divisor $D$ on $Z$ such that $D|_{Z_K} \sim_{\mathbb{Q}} H$. Note in particular, $D$ is big over $R$.

By the exact sequence we deduce that on $Z$ we have $N-\psi^*D \equiv \sum b_i E_i$ for certain $b_i \in \mathbb{Z}$.
By adding a sufficiently high multiple of the central fibre $X_k$, we see that $N \equiv \psi^*D + F$, where $F$ is effective.
As $D$ is big over $R$, we conclude that $N \equiv E$ for some $E$ effective. 
As $N=\pi^*L$, we conclude $L \equiv \pi_*E$. 
\end{proof}

\bibliographystyle{amsalpha}
\bibliography{refs}
	
\end{document}